\documentclass[11pt]{article}
\usepackage{epic,latexsym,amssymb}
\usepackage{color}
\usepackage{soul}
\usepackage{hyperref}
\usepackage{tikz}
\usepackage{natbib}
\usepackage{bbm}
\usetikzlibrary{decorations.pathreplacing}
\usepackage{amsfonts,epsf,amsmath,comment}
\usepackage{lineno}

\usepackage{subfig}
\usepackage{epic}
\usepackage{amsthm}
\usepackage{latexsym}
\usepackage{tikz}
\usetikzlibrary{decorations.markings}

\usetikzlibrary{arrows}

\newtheorem{theorem}{Theorem}
\newtheorem{lemma}{Lemma}

\newcommand{\Exp}{\mathbb{E}}

\newcommand{\tcl}{\mathcal{T}}
\newcommand{\bigO}{{\rm O}}
\newcommand\dto{\overset{\mathrm{d}}{\to}}

\newcommand{\gpr}{\gamma_{\rm pr}}
\newcommand{\mpr}{\mu_{\rm pr}}
\newcommand{\spr}{\sigma_{\rm pr}}

\newcommand{\2}{\vspace{0.2cm}}

\newenvironment{unnumbered}[1]{\trivlist \item [\hskip \labelsep {\bf
#1}]\ignorespaces\it}{\endtrivlist}

\newcommand{\parent}{\mbox{{\rm parent}}}

\newcommand{\lineN}[1]{\mbox{#1.} \hspace{0.1cm}}
\newcommand{\trueX}{\mbox{true}}
\newcommand{\falseX}{\mbox{false}}

\newcommand{\InPD}{\mbox{{\rm In\_PD\_Set}}}
\newcommand{\PDchild}{\mbox{{\rm Dom\_by\_child}}}

\let\oldenumerate\enumerate
\renewcommand{\enumerate}{
  \oldenumerate
  \setlength{\itemsep}{1pt}
  \setlength{\parskip}{0pt}
  \setlength{\parsep}{0pt}
}

\def\vertex(#1){\put(#1){\circle*{2}}}
\def\vertexo(#1){\put(#1){\circle{2}}}
\def\vert(#1){\put(#1){\circle*{1.5}}}
\def\verto(#1){\put(#1){\circle{1.5}}}
\def\lab(#1)#2{\put(#1){\makebox(0,0)[c]{#2}}}
\begin{document}

\title{Paired domination in trees: \\ A linear algorithm and asymptotic normality}

\author{$^{1,2}$Michael A. Henning \, and \, $^3$Dimbinaina Ralaivaosaona \\
\\
$^1$Department of Mathematics and Applied Mathematics\\
University of Johannesburg \\
Auckland Park, 2006, South Africa \\
\small \tt Email: mahenning@uj.ac.za \\
\\
$^2$DSTI-NRF Centre of Excellence \\
in Mathematical and Statistical
Sciences (CoE-MaSS), South Africa \\
\\
$^3$Department of Mathematical Sciences \\
Stellenbosch University \\
Matieland, 7602, South Africa \\
\small \tt Email: naina@sun.ac.za
}

\date{}
\maketitle

\begin{abstract}
A set $S$ of vertices in a graph $G$ is a paired dominating set if every vertex of $G$ is adjacent to a vertex in $S$ and the subgraph induced by $S$ contains a perfect matching (not necessarily as an induced subgraph). The paired domination number, $\gpr(G)$, of $G$ is the minimum cardinality of a paired dominating set of $G$. We present a linear algorithm for computing the paired domination number of a tree. As an application of our algorithm, we prove that the paired domination number is asymptotically normal in a random rooted tree of order $n$ generated by a conditioned Galton--Watson process as $n\to\infty$. In particular, we have found that the paired domination number of a random Cayley tree of order $n$, where each tree is equally likely, is asymptotically normal with expectation approaching $(0.5177\ldots)n$.
\end{abstract}


\section{Introduction}

A \emph{dominating set} of a graph $G$ is a set $S \subseteq V(G)$ such that every vertex of $V(G)\setminus S$ is adjacent to some vertex in $S$. A set $S$ of vertices in $G$ is a \emph{total dominating set}, abbreviated TD-set,  of $G$ if every vertex in $G$ is adjacent to some other vertex in $S$. The \emph{domination number}, $\gamma(G)$, of $G$ is the minimum cardinality of a dominating set of $G$, and the \emph{total domination number}, $\gamma_t(G)$, of $G$ is the minimum cardinality of a TD-set of $G$. If $A$ and $B$ are two subsets of vertices in a graph $G$, then the set $A$ \emph{totally dominates} the set $B$ is every vertex in $B$ has a neighbor in the set $A$, where two vertices are neighbors if they are adjacent. For recent books on domination and total domination in graphs, we refer the reader to~\cite{HaHeHe-20,HaHeHe-21,HaHeHe-23,HeYe-book}.

A \emph{paired dominating set}, abbreviated PD-set, of an isolate-free graph $G$ is a dominating set $S$ of $G$ with the additional property that the subgraph $G[S]$ induced by $S$ contains a perfect matching $M$ (not necessarily induced). With respect to the matching $M$, two vertices joined by an edge of $M$ are \emph{paired} and are called \emph{partners} in $S$. The \emph{paired domination number}, $\gpr(G)$, of $G$ is the minimum cardinality of a PD-set of $G$. We call a PD-set of $G$ of cardinality $\gpr(G)$ a $\gpr$-\emph{set of $G$}. We note that the paired domination number $\gpr(G)$ is an even integer. For a recent survey on paired domination in graphs, we refer the reader to the book chapter~\cite{DeHaHe-20}.

We denote the \emph{degree} of a vertex $v$ in a graph $G$ by $\deg_G(v)$. A vertex of degree $0$ is called an \emph{isolated vertex}, and a graph is \emph{isolate}-\emph{free} if it contains no isolated vertex. The maximum (minimum) degree among the vertices of $G$ is denoted by $\Delta(G)$ ($\delta(G)$, respectively). A \emph{leaf} of a tree $T$ is a vertex of degree~$1$ in $T$, and a \emph{support vertex} of $T$ is a vertex with a leaf neighbor. The \emph{distance} $d(u,v)$ between two vertices $u$ and $v$ in a connected graph $G$, equals the minimum length of a $(u,v)$-path in $G$ from $u$ to $v$.

A \emph{rooted tree} $T$ distinguishes one vertex $r$ called the \emph{root}. For each vertex $v \ne r$ of $T$, the \emph{parent} of $v$ is the neighbor of $v$ on the unique $(r,v)$-path, while a \emph{child} of $v$ is any other neighbor of $v$. A \emph{descendant} of $v$ is a vertex $u \ne v$ such that the unique $(r,u)$-path contains $v$. We let $D(v)$ denote the set of descendants of $v$, and we define $D[v] = D(v) \cup \{v\}$. The \emph{maximal subtree} at $v$ is the subtree of $T$ induced by $D[v]$, and is denoted by $T_v$.

The aim of this paper is twofold: First, we present a simple bottom-up linear-time algorithm that solves the minimum paired domination problem for trees and prove its correctness. Our algorithm has a similar flavour to that of Mitchell, Cockayne, and Hedetniemi~\cite{MiCoHe-79} for the domination number problem.  Second, we establish that the paired domination number in random trees of order $n$ admits a nondegenerate Gaussian limit law as $n \to \infty$. The random trees considered here are conditioned Galton--Watson trees, and our bottom-up algorithm plays a crucial role in proving the result on the limiting distribution.

This paper is organised as follows: In Section~\ref{sec:2}, we focus on our new algorithm and prove its correctness. In Section~\ref{sec:3}, we analyse the limiting distribution of the paired domination number in conditioned Galton--Watson trees in general and provide explicit examples and simulations that are in support of our result.

\section{A paired domination tree algorithm}\label{sec:2}
\label{S:algorithm}

In 1979 Mitchell, Cockayne, and Hedetniemi~\cite{MiCoHe-79} and in 1984 Laskar, Pfaff, Hedetniemi, and Hedetniemi~\cite{LaPfHeHe-84} presented linear algorithms for computing the domination number and, respectively, the total domination number, of a tree. Their algorithms consider an arbitrary tree, which is rooted, and they systematically consider the vertices of the tree, starting from the vertices at furthest distance from the root, and carefully select a minimum (total) dominating set $S$ in such a way that the sum of the distances from the root to the vertices in $S$ is a minimum.

In 2003 Qiao, Kang, Cardei, and Du~\cite{QiKaCaDu-03} presented a linear-time algorithm for computing the paired domination number of a tree. The linear-time algorithm we present in this section has a similar flavour to the linear algorithms presented in~\cite{LaPfHeHe-84,MiCoHe-79}. In order to formally state our algorithm for computing the paired domination number of a tree, we first introduce the necessary notation.

Let $T$ be a rooted tree. We commonly draw the root $r$ of $T$ at the top with the remaining vertices at the appropriate level below $r$ depending on their distance from~$r$. We label the root with label~$1$. We label all children of the root (at level~$1$ from the root) next with labels $2, \ldots, \deg(1) + 1$. Next we label all children of the vertex~$2$ with labels $\deg(1) + 2, \ldots, \deg(1) + \deg(2) + 1$, and thereafter we label all children of the vertex~$3$, etc. Given such a rooted tree $T$ with root vertex labelled $1$ and with $V(T) = \{1,2,\ldots,n\}$, we represent $T$ by a data structure called a \emph{Parent array} in which the parent of a vertex labelled $i$ is given by $\parent{}[i]$ with $\parent[1] = 0$ (to indicate that the vertex labelled $1$ has no parent). We assume that the vertices of $T$ are labelled $1,2,\ldots,n$ so that for $i < j$, vertex $i$ is at level less than or equal to that of vertex $j$ (that is, $d(1,i) \le d(1,j)$). Figure~\ref{fig:root} shows an example of a rooted tree $T$\index{tree} with its parent array.

\begin{figure}[htb]
\begin{center}
\begin{tikzpicture}[scale=.8,style=thick,x=1cm,y=1cm]
\def\vr{2.5pt} 
\path (0,1) coordinate (v10);
\path (0,2) coordinate (v5);
\path (0,3) coordinate (v2);
\path (2,1) coordinate (v11);
\path (2,2) coordinate (v6);
\path (3,0) coordinate (v14);
\path (4,1) coordinate (v12);
\path (4,2) coordinate (v7);
\path (4,3) coordinate (v3);
\path (4,3.2) coordinate (v3p);
\path (4,4.5) coordinate (v1);
\path (5,0) coordinate (v15);
\path (6,1) coordinate (v13);
\path (6,2) coordinate (v8);
\path (8,2) coordinate (v9);
\path (8,3) coordinate (v4);
\draw (v10)--(v5)--(v2)--(v1)--(v4)--(v9);
\draw (v1)--(v3)--(v7)--(v12)--(v14);
\draw (v12)--(v15);
\draw (v11)--(v6)--(v3)--(v8)--(v13);
\draw (v1) [fill=white] circle (\vr);
\draw (v2) [fill=white] circle (\vr);
\draw (v3) [fill=white] circle (\vr);
\draw (v4) [fill=white] circle (\vr);
\draw (v5) [fill=white] circle (\vr);
\draw (v6) [fill=white] circle (\vr);
\draw (v7) [fill=white] circle (\vr);
\draw (v8) [fill=white] circle (\vr);
\draw (v9) [fill=white] circle (\vr);
\draw (v10) [fill=white] circle (\vr);
\draw (v11) [fill=white] circle (\vr);
\draw (v12) [fill=white] circle (\vr);
\draw (v13) [fill=white] circle (\vr);
\draw (v14) [fill=white] circle (\vr);
\draw (v15) [fill=white] circle (\vr);
\draw[anchor = south] (v1) node {{\small $1$}};
\draw[anchor = east] (v2) node {{\small $2$}};
\draw[anchor = east] (v5) node {{\small $5$}};
\draw[anchor = east] (v10) node {{\small $10$}};
\draw[anchor = east] (v6) node {{\small $6$}};
\draw[anchor = east] (v11) node {{\small $11$}};
\draw[anchor = west] (v3p) node {{\small $3$}};
\draw[anchor = west] (v7) node {{\small $7$}};
\draw[anchor = west] (v12) node {{\small $12$}};
\draw[anchor = west] (v15) node {{\small $15$}};
\draw[anchor = east] (v14) node {{\small $14$}};
\draw[anchor = west] (v8) node {{\small $8$}};
\draw[anchor = west] (v13) node {{\small $13$}};
\draw[anchor = west] (v4) node {{\small $4$}};
\draw[anchor = west] (v9) node {{\small $9$}};
\end{tikzpicture}
\end{center}
\[
\begin{array}{crcccccccccccccccl}
& & 1 & 2 & 3 & 4 & 5 & 6 & 7 & 8 & 9 & 10 & 11 & 12 & 13 & 14 & 15 &
\\
{\rm Parent} & [ & 0 & 1 & 1 & 1 & 2 & 3 & 3 & 3 & 4 & 5 & 6 & 7 & 8 &
12 & 12 & ]
\end{array}
\]
\caption{A rooted tree $T$ with its Parent array}
\label{fig:root}
\end{figure}
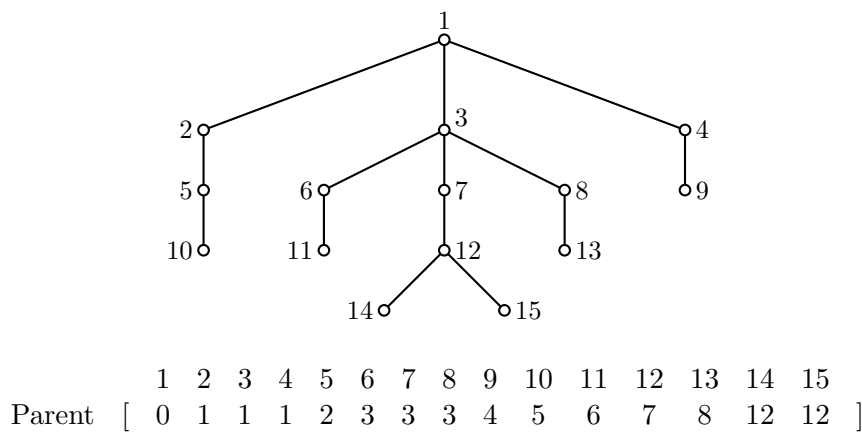

In Algorithm TREE PAIRED DOMINATION that follows, we will for each vertex, $i$, keep track of two boolean values, namely
\[
\InPD{}[i] \hspace*{0.5cm} \mbox{and} \hspace*{0.5cm} \PDchild[i].
\]

These values are initially set to $\falseX{}$.  As the algorithm progresses some of these values may change to true. Upon completion of the algorithm, all vertices, $i$, with $\InPD{}[i]=\trueX{}$ will be added to the PD-set. Further, all vertices with $\PDchild{}[i]=\trueX{}$ will be dominated by one of their children. We define vertex $0$ to be a \emph{dummy}-\emph{vertex} that does not exist and we define $\parent{}[1]=\parent{}[0]=0$. We will never change $\InPD{}[0]$ but may change $\PDchild{}[0]$.

\begin{unnumbered}{Algorithm TREE PAIRED DOMINATION} \textbf{:} \\ \\
\textbf{Input:} {\rm A rooted tree $T$ with $V(T) = \{1,2,\ldots,n\}$ rooted at $1$ where $n \ge 2$ (and where larger } \\
\hspace*{1.01cm} {\rm  values are further from the root) and
represented by an array $\parent[1 \ldots n]$ } \\
\textbf{Output:} {\rm An array $\InPD{}[]$ that indicates which vertices belong to the PD-set} \\
\textbf{Code:}
\begin{enumerate}
\item[] \lineN{1} {\rm \textbf{for} $i = 0$ \textbf{to} $n$ \textbf{do} $\{$ } \\
\lineN{2}         \hspace*{0.8cm} $\PDchild{}[i] = \falseX{}$ \\
\lineN{3}         \hspace*{0.8cm} $\InPD{}[i] = \falseX{}$ $\}$
\item[] \lineN{4} \textbf{for} $i = n$ \textbf{to} $2$ \textbf{do} \\
\item[]
\lineN{5}         \hspace*{0.8cm} \textbf{if} {\rm (}$\PDchild{}[i]=\falseX{}${\rm )}
\item[]           \hspace*{2cm} \textbf{and} {\rm (}$\InPD{}[\parent{}[i]]=\falseX{}${\rm )}
\item[]           \hspace*{2cm} \textbf{and} {\rm (}$\parent{}[i] \ne 1${\rm )}
\item[]           \hspace*{2cm} \textbf{and} {\rm (}$\InPD{}[\parent{}[\parent{}[i]]]=\falseX{}${\rm )}
\item[]
\item[]           {\rm (}we remark that in this case, no child of vertex $i$ is in the PD-set, and neither is the parent of vertex~$i$ nor the parent of the parent of vertex~$i$; further, the parent of vertex~$i$ is not the vertex~$1${\rm )}
\item[]
\item[]           \hspace*{1.5cm} \textbf{then} $\{$ \\
\lineN{6}         \hspace*{1.6cm} $\InPD{}[\parent{}[i]] =\trueX{}$ \\
\lineN{7}         \hspace*{1.6cm} $\InPD{}[\parent{}[\parent{}[i]]] =\trueX{}$ \\
\lineN{8}         \hspace*{1.6cm} $\PDchild{}[\parent{}[\parent{}[i]]] = \trueX$ \\
\lineN{9}         \hspace*{1.6cm} $\PDchild{}[\parent{}[\parent{}[\parent{}[i]]]] = \trueX$  (if the vertex exists)$\}$
\item[]
\item[]
{\rm (}we remark that in this case, we add the parent of vertex~$i$ and the parent of the parent of vertex~$i$ to the PD-set, and these two added vertices are paired in the PD-set. Further, we note that the parent of the parent of vertex~$i$ now is dominated by one of its children, and so we update $\PDchild{}[\parent{}[i]${\rm )}  \\
\item[]
\lineN{10}         \hspace*{0.8cm} \textbf{if} {\rm (}$\PDchild{}[i]=\falseX{}${\rm )}
\item[]           \hspace*{2cm} \textbf{and} {\rm (}$\InPD{}[\parent{}[i]]=\falseX{}${\rm )}
\item[]           \hspace*{2cm} \textbf{and} {\rm (}$\parent{}[i] = 1${\rm )} \textbf{or} {\rm (}$\InPD{}[\parent{}[\parent{}[i]]]=\trueX{}${\rm )}
\item[]
\item[]           {\rm (}we remark that in this case, the PD-set contains no child of vertex $i$  and does not contain the parent of vertex~$i$, but either the parent of vertex~$i$ is the vertex~$1$ or the PD-set contains the parent of the parent of vertex~$i$.{\rm )}
\item[]
\item[]           \hspace*{1.5cm} \textbf{then} $\{$ \\
\lineN{11}         \hspace*{1.6cm} $\InPD{}[\parent{}[i]] =\trueX{}$ \\
\lineN{12}         \hspace*{1cm} \textbf{and if} $i_1$ is the child of $\parent{}[i]$ of largest label such that $\InPD{}[i_1] = \falseX{}$
\item[]           \hspace*{2.4cm} \textbf{then} $\{$ \\
\lineN{13}         \hspace*{3cm} $\InPD{}[i_1] = \trueX{}$ \\
\lineN{14}         \hspace*{3cm} $\PDchild{}[\parent{}[i]] = \trueX$ $\}$ $\}$
\item[]
\item[]           {\rm (}we remark that in this case, we add the parent of vertex~$i$ and a child of the parent of vertex~$i$ of the largest label to the PD-set, and these two added vertices are paired in the PD-set. Further, we note that the parent of vertex~$i$ now is dominated by one of its children, and so we update $\PDchild{}[\parent{}[i]${\rm )}  \\
\item[]
\lineN{15} \textbf{if} {\rm (}$\PDchild{}[1]=\falseX{}${\rm )}
\item[]           \hspace*{1.5cm} \textbf{then} $\{$ \\
\lineN{16} \hspace*{2.4cm} $\InPD{}[1] =\trueX{}$ \\
\lineN{17} \hspace*{2.4cm} $\InPD{}[2] =\trueX{}$ \\
\lineN{18}  \hspace*{2.4cm} $\PDchild{}[1] = \trueX$ $\}$
\item[]
\item[]           {\rm (}we remark that in this case, no child of the vertex~$1$ is in the PD-set, and so we add to the PD-set the vertices~$1$ and~$2$, and these two added vertices are paired in the PD-set{\rm )}
\end{enumerate}
\end{unnumbered}

We prove next that algorithm {\small \textbf{TREE PAIRED DOMINATION}} produces a $\gpr$-set of $T$.

\begin{theorem}
\label{tree:alg}
If $T$ is tree of order~$n \ge 2$, then the set $S$ produced by algorithm {\small
\textbf{TREE PAIRED DOMINATION}} defined by
\[
S = \{ i \in [n] \colon \InPD{}[i] =\trueX{} \}
\]
is a $\gpr$-set of $T$.
\end{theorem}
\begin{proof} Let $S$ be the set produced by algorithm {\small \textbf{TREE PAIRED DOMINATION}}, and so $S = \{ i \in [n] \colon \InPD{}[i] =\trueX{} \}$. By construction, $S$ is a PD-set of $T$. Hence it remains to show that $S$ is a $\gpr$-set of $T$. Let $S(i)$ denote all vertices, $j$, with $\InPD{}[j]=\trueX{}$, immediately after performing line~$14$ in our tree algorithm with the value $i$. We show that property $P(i)$ below holds for all $i=n,n-1,n-2,\ldots,2$, by induction.

\2
\noindent \textbf{Property $P(i)$:} The following properties hold. \\ [-22pt]
\begin{enumerate}
\item[1.] The set $S(i)$ totally dominates the set $\{i,i+1,i+2,\ldots,n\}$.
\item[2.] The set $S(i) \subseteq Q_i$ for some $\gpr$-set $Q_i$ of $T$.
\end{enumerate}

\2

We first consider Property $P(n)$ (and so here $i = n$). If $\parent{}[n] = 1$, then $T$ is a star rooted at its central vertex and $S(n) = \{1,2\}$ and $\PDchild{}[1]=\trueX{}$. If $\parent{}[n] \ne 1$, then $S(n) = \{\parent{}[n],\parent{}[\parent{}[n]]\}$ and $\PDchild{}[\parent{}[\parent{}[n]]] = \trueX$. In both cases, the set $S(n)$ contains the parent of vertex~$n$, and therefore $S(n)$ totally dominates the set $\{n\}$. Let $Q_n$ be an arbitrary $\gpr$-set of $T$. Moreover since the parent of vertex~$n$ is a support vertex (with leaf neighbor~$n$), every $\gpr$-set of $T$ contains the parent of vertex~$n$. In particular, $\parent{}[n] \in Q_n$.

Suppose that $\parent{}[n] = 1$. In this case, $\gpr(T) = 2$ and $Q_n$ contains a leaf neighbor of $\parent{}[n]$. If vertex~$2$ does not belong to~$Q_n$, then replacing the leaf neighbor of $\parent{}[n]$ that belongs to $Q_n$ with the vertex~$2$ produces a $\gpr$-set of $T$ that is identically the set $S(n)$. Suppose that $\parent{}[n] \ne 1$. In this case, if $\parent{}[\parent{}[n]]$ does not belong to the set $Q_n$, then replacing the leaf neighbor of $\parent{}[n]$ that belongs to $Q_n$ with the vertex~$\parent{}[\parent{}[n]]$ produces a new $\gpr$-set of $T$ that contains the set $S(n)$. Furthermore, the array $\PDchild{}$ is updated correctly noting that in both cases $\PDchild{}[\parent{}[n]]=\trueX{}$ upon completion of line~$14$ in our tree algorithm with the value $n$. Thus, property $P(n)$ holds. This establishes the base case.

For the inductive hypothesis, let $2 \le i < n$ and assume that property $P(i+1)$ holds. Thus, the set $S(i+1)$ totally dominates the set $\{i+1,i+2,\ldots,n\}$. Moreover, $S(i+1) \subseteq Q_{i+1}$ for some $\gpr$-set $Q_{i+1}$ of $T$, and the array $\PDchild{}$ is updated correctly. For notation convenience, let $p_i$ be the parent of vertex~$i$ and if $p_i \ne 1$, then let $g_i$ be the grandparent of vertex~$i$. Thus, $p_i = \parent{}[i]$ and $g_i = \parent{}[\parent{}[i]]$.

We now show that property $P(i)$ holds. Consider when the vertex~$i$ is first considered in the algorithm. Thus at this stage, the vertex~$i+1$ has already been analyzed by the algorithm and the next vertex to be considered is vertex~$i$.  If $\PDchild{}[i]=\trueX{}$ or $\InPD{}[p_i]=\trueX{}$, then the vertex $i$ is totally dominated by $S(i+1)$ and therefore $S(i+1)$ totally dominates $\{i,i+1,i+2,\ldots,n\}$. In this case, we let $S(i) = S(i+1)$ and $Q_i = Q_{i+1}$ and note that $S(i) \subseteq Q_{i+1} = Q_i$, which implies that $S(i)$ is a subset of a $\gpr$-set of $T$, namely $Q_i$.  Hence we may assume that $\PDchild{}[i]=\falseX{}$ and $\InPD{}[p_i]=\falseX{}$, for otherwise Property $P(i)$ holds as desired.

We now consider the $\gpr$-set, $Q_{i+1}$, of $T$ that contains the set $S(i+1)$. We note that the set $Q_{i+1}$ exists by the inductive hypothesis. Some vertex, $t$, in $Q_{i+1}$ must totally dominate vertex~$i$. Thus either $t$ is the parent of vertex~$i$ or a child of vertex~$i$. Let $t'$ be the vertex paired with $t$ in the set $Q_{i+1}$. We consider next the two possible outcomes of lines~5 and~10 when the vertex~$i$ is the current vertex considered.

\medskip
\emph{Case~1. The statements in line~5 hold; that is, } \\ [-20pt]
\begin{enumerate}
\item[$\bullet$] $\PDchild{}[i]=\falseX{}$,
\item[$\bullet$] $\InPD{}[p_i]=\falseX{}$,
\item[$\bullet$] $p_i \ne 1$, and $\InPD{}[g_i]=\falseX{}$.
\end{enumerate}

Thus the set $S(i+1)$ contains no child of vertex~$i$, and neither $p_i$ nor $g_i$ belong to $S(i+1)$. According to lines~$6$ and~$7$, we have $S(i) = S(i+1) \cup \{p_i,g_i\}$.

Suppose firstly that $\{p_i,g_i\} \subseteq Q_{i+1}$. In this case, we let $Q_i = Q_{i+1}$ and note that $S(i) \subseteq Q_i$. Moreover since $p_i \in S(i)$, the set $S(i)$ totally dominates the set $\{i,i+1,i+2,\ldots,n\}$. Since $\PDchild{}[g_i] = \trueX$ in line~$8$, the array $\PDchild{}$ is updated correctly. Hence, we may assume that $p_i \notin Q_{i+1}$ or $g_i \notin Q_{i+1}$ (or both $p_i \notin Q_{i+1}$ and $g_i \notin Q_{i+1}$).

Suppose secondly that $p_i \notin Q_{i+1}$, implying that $t$ is a child of $i$. Let $T_i$ be the subtree of $T$ rooted at the vertex~$i$. By the inductive hypothesis, the set $S(i+1)$ totally dominates the set $\{i+1,i+2,\ldots,n\}$. By supposition, $\PDchild{}[i]=\falseX{}$, and so no child of vertex~$i$ belongs to the set $S(i+1)$. Thus the set $S(i+1)$ totally dominates all vertices in the subtree $T_i$, except for the vertex~$i$. By construction, the set $S(i+1)$ is therefore a PD-set of $T_i - i$. Let $S_i' = V(T_i) \cap S(i+1)$, and let $Q_i' = V(T_i) \cap Q_{i+1}$. Since $S(i+1) \subset Q_{i+1}$ and $t \in Q_{i+1} \setminus S(i+1)$, we note that $S_i' \subset Q_i'$, implying that $|S_i'| \le |Q_i'| - 2$. If $g_i \notin Q_{i+1}$, then we let $Q_i = (Q_{i+1} \setminus Q_i') \cup (S_i' \cup \{p_i,g_i\})$. If $g_i \in Q_{i+1}$, then we let $Q_i = (Q_{i+1} \setminus Q_i') \cup (S_i' \cup \{p_i,i\})$. In both cases, $S(i) \subseteq Q_i$. Further, the set $Q_i$ is a PD-set of cardinality~$|Q_i| = |Q_{i+1}| - |Q_i'| + (|S_i'| + 2) \le |Q_{i+1}| = \gpr(T)$, and is therefore a $\gpr$-set of $T$. Hence if $p_i \notin Q_{i+1}$, then we have shown that the set $S(i)$ totally dominates the set $\{i,i+1,i+2,\ldots,n\}$ and is contained in some $\gpr$-set of $T$, namely the set $Q_i$ defined above.

Suppose thirdly that $p_i \in Q_{i+1}$ and  $g_i \notin Q_{i+1}$. In this case we can choose the vertex~$t$ to be the vertex~$p_i$, that is, $t = p_i$. Let $t'$ be the vertex paired with $t$ in the PD-set $Q_{i+1}$. Since $g_i \notin Q_{i+1}$, the vertex $t'$ is a child of $p_i$. Possibly, $t' = i$. If $t'$ is a leaf, then we let $Q_i = (Q_{i+1} \setminus \{t'\}) \cup \{g_i\}$ with $p_i$ and $g_i$ paired in $Q_i$, and with all other pairings in $Q_i$ the same as the pairings in $Q_{i+1} \setminus \{t,t'\}$. As before, $S(i) \subseteq Q_i$ and the set $S(i)$ totally dominates the set $\{i,i+1,i+2,\ldots,n\}$. Hence, we may assume that $t'$ is not a leaf. Let $T'$ be the subtree of $T$ rooted at the vertex~$t'$, and let $j$ be an arbitrary child of $t'$. We note that $j \ge i+1$, and so by the inductive hypothesis the set $S(j)$ totally dominates all vertices in the subtree $T_{j}$ rooted at vertex~$j$. Since $S(j) \subseteq S(i+1)$, we infer that all vertices in $T'$ are totally dominated by $S(i+1)$, except possibly for the vertex~$t'$. Hence as before, we let $Q_i = (Q_{i+1} \setminus \{t'\}) \cup \{g_i\}$ with $p_i$ and $g_i$ paired in $Q_i$, to produce the desired result.

\medskip
\emph{Case~2. The statements in line~10 hold; that is,} \\ [-20pt]
\begin{enumerate}
\item[$\bullet$] $\PDchild{}[i]=\falseX{}$,
\item[$\bullet$] $\InPD{}[p_i]=\falseX{}$,
\item[$\bullet$] $p_i = 1$ or $p_i \ne 1$ and $\InPD{}[g_i] = \trueX{}$.
\end{enumerate}

Thus, the set $S(i+1)$ does not contain the parent $p_i$ of vertex~$i$ and contains no child of vertex~$i$. However, either $p_i = 1$ or $p_i \ne 1$ and $\InPD{}[g_i]=\trueX{}$. By the inductive hypothesis, $S(i+1) \subseteq Q_{i+1}$. In particular, if $p_i \ne 1$, then we infer that $g_i \in Q_{i+1}$. According to lines~$11$--$14$, we have $S(i) = S(i+1) \cup \{p_i,i_1\}$ where $i_1$ is the child of $\parent{}[i]$ of largest label such that $\InPD{}[i_1] = \falseX{}$. Possibly, $i_1 = i$.

Suppose firstly that $\{p_i,i_1\} \subseteq Q_{i+1}$. In this case, we let $Q_i = Q_{i+1}$ and note that $S(i) \subseteq Q_i$. Moreover since $p_i \in S(i)$, the set $S(i)$ totally dominates the set $\{i,i+1,i+2,\ldots,n\}$. Since $\PDchild{}[p_i] = \trueX$ in line~$14$, the array $\PDchild{}$ is updated correctly. Hence, we may assume that $p_i \notin Q_{i+1}$ or $i_1 \notin Q_{i+1}$ (or both $p_i \notin Q_{i+1}$ and $i_1 \notin Q_{i+1}$).

Suppose secondly that $p_i \notin Q_{i+1}$, implying that $t$ is a child of $i$. Let $T_i$ be the subtree of $T$ rooted at the vertex~$i$. By the inductive hypothesis, the set $S(i+1)$ totally dominates the set $\{i+1,i+2,\ldots,n\}$. By supposition, no child of vertex~$i$ belongs to the set $S(i+1)$. Thus the set $S(i+1)$ totally dominates all vertices in the subtree $T_i$, except for the vertex~$i$. We proceed now analogously as in Case~1. Adopting our notation in Case~1, let $S_i' = V(T_i) \cap S(i+1)$ and $Q_i' = V(T_i) \cap Q_{i+1}$. In this case, $|S_i'| \le |Q_i'| - 2$. If $i_1 \notin Q_{i+1}$, then we let $Q_i = (Q_{i+1} \setminus Q_i') \cup (S_i' \cup \{p_i,i_1\})$. If $i_1 \in Q_{i+1}$, then we let $Q_i = (Q_{i+1} \setminus Q_i') \cup (S_i' \cup \{p_i,i\})$. In both cases, $S(i) \subseteq Q_i$, where $Q_i$ is a $\gpr$-set of $T$.

Suppose thirdly that $p_i \in Q_{i+1}$ and $i_1 \notin Q_{i+1}$. In this case we choose the vertex~$t$ to be the vertex~$p_i$, that is, $t = p_i$. Let $t'$ be the vertex paired with $t$ in the PD-set $Q_{i+1}$.

Suppose that the vertex $t'$ is a child of $p_i$. Possibly, $t' = i$. If $t'$ is a leaf, then we let $Q_i = (Q_{i+1} \setminus \{t'\}) \cup \{i_1\}$ with $p_i$ and $i_1$ paired in $Q_i$, and with all other pairings in $Q_i$ the same as the pairings in $Q_{i+1} \setminus \{t,t'\}$. As before, $S(i) \subseteq Q_i$ and the set $S(i)$ totally dominates the set $\{i,i+1,i+2,\ldots,n\}$. Hence, we may assume that $t'$ is not a leaf. Let $T'$ be the subtree of $T$ rooted at the vertex~$t'$, and let $j$ be an arbitrary child of $t'$. We note that $j \ge i+1$, and so by the inductive hypothesis the set $S(j)$ totally dominates all vertices in the subtree $T_{j}$ rooted at $j$. Since $S(j) \subseteq S(i+1)$, we infer that all vertices in $T'$ are totally dominated by $S(i+1)$, except possibly for the vertex~$t'$. Hence as before, we let $Q_i = (Q_{i+1} \setminus \{t'\}) \cup \{i_1\}$ with $p_i$ and $i_1$ paired in $Q_i$, to produce the desired result.

Hence we may assume that the vertex $t'$ cannot be chosen to be a child of $p_i$, implying that $t' = g_i$. According to our algorithm, since $g_i \in Q_{i+1}$ and $p_i \notin S(i+1)$, the vertex $g_i$ was added to the set $S(i+1)$ when a grandchild $j$ of $g_i$ was considered in line~$5$ of the algorithm for some $j \ge i+1$. In this case, denoting the parent of vertex~$j$ by $p_j$, and so $p_j = \parent{}[j]$, our tree algorithm added to the set $S(j+1)$ the vertices $p_j$ and $g_i$ to produce the set $S(j)$ upon completion of line~$14$ in our tree algorithm with the value $j$. By construction of the set $S$, we note that $S(j) \subseteq S(i+1)$, and so $\{p_j,g_i\} \subseteq Q_{i+1}$.

By assumption, the vertices $t = p_i$ and $t' = g_i$ are paired in $Q_{i+1}$. Let $p_j'$ be the vertex paired with $p_j$ in the set $Q_{i+1}$. Since the parent, $g_i$, of vertex $p_j$ is paired with vertex $p_i$, the vertex $p_j'$ is a child of $p_j$. Let $T_j$ be the subtree of $T$ rooted at the vertex~$j$. By the inductive hypothesis, the set $S(j)$, which is contained in the set $S(i+1)$, totally dominates all descendants of the vertex~$j$, implying that the vertex $p_j'$ is not needed to totally dominate any descendant of the vertex~$j$. Hence, the vertex~$p_j'$ is needed only to partner the vertex~$p_j$. We let $Q_i = (Q_{i+1} \setminus \{g_j\}) \cup \{i_1\}$, where the vertices $i_1$ and $p_i$ are paired and the vertices $p_j$ and $g_i$ are paired and where all other pairing of vertices in $Q_{i+1}$ remain unchanged. The resulting set $Q_i$ is a PD-set of $T$ satisfying $|Q_i| = |Q_{i+1}|$. Furthermore, $S(i) \subseteq Q_i$. Thus, $Q_i$ is a $\gpr$-set of $T$ containing the set $S(i)$. Thus, property $P(i)$ also holds in Case~2.

By Cases~1 and~2, property $P(i)$ holds for all $i=n,n-1,n-2,\ldots,2$, by induction. In particular, property~$P(2)$ holds, and so the set $S(2)$ totally dominates the set $\{2,3,\ldots,n\}$. Moreover, the set $S(2) \subseteq Q_2$ for some $\gpr$-set $Q_2$ of $T$. Thus, $S(2)$ is a subset of an optimal solution $Q_2$, and $S(2)$ totally dominates all vertices in $G$, except possibly for vertex~$1$. Therefore if vertex~$1$ is totally dominated by $S(2)$, then $S(2) = Q_2$ and $S(2)$ is an optimal solution. In this case, we take $Q_1 = Q_2 = S_2$, and the desired result follows.

Hence we may assume that vertex~$1$ is not totally dominated by $S(2)$. In this case, some vertex, $z$, in $Q_{2}$ must totally dominate vertex~$1$. Thus, $z$ is a child of vertex~$1$. Since $z \notin S(2)$ and both $|S(2)|$ and $|Q_2|$ are even, we infer that $|S(2)| \le |Q_2| - 2$. We now let $Q_1 = S(2) \cup \{1,2\}$ with the vertices $1$ and $2$ paired in $Q_1$ and with all other pairing in $S(2)$ unchanged. The resulting set $Q_1$ is a PD-set of $T$. We infer that $|Q_2| = \gpr(T) \le |Q_1| = |S(2)| + 2 \le |Q_2| = \gpr(T)$, and so we must have equality throughout this inequality chain. In particular, $|Q_1| = \gpr(T)$, implying that $Q_1$ is a $\gpr$-set of $T$. Furthermore, $S(1) \subseteq Q_1$. Thus, $Q_1$ is a $\gpr$-set of $T$ containing the set $S(1)$. Thus, property $P(1)$ holds. Since $S(1)$ is a PD-set of $T$ and $|S(1)| \le |Q_1| = \gpr(T)$, the set $S(1)$ is an optimal solution, that is, $S(1)$ is a $\gpr$-set of $T$. Since $S = S(1)$, our algorithm {\small \textbf{TREE PAIRED DOMINATION}} does indeed find an optimal solution. This completes the proof of Theorem~\ref{tree:alg}.
\end{proof}

\medskip
We remark that the time complexity of the algorithm is $\bigO(n)$ for trees of order $n$. When we run our algorithm on the rooted tree $T$ in Figure~\ref{fig:root} the following steps will be performed, yielding the $\gpr$-set $\{12,11,8,7,6,5,4,3,2,1\}$ as a minimum PD-set in $T$. In particular, we note that $\gpr(T) = 10$.

\2
\begin{table}[htb]
    \centering
\begin{tabular}{|c|l|r|c|l|r|} \cline{1-2} \cline{4-5}
$i$ & action  &  \hspace{1cm} & $i$ & action \\  \cline{1-2} \cline{4-5}
15  & $\InPD{}[7]=\InPD{}[12]=\trueX{}$  &  & 7  & none \\
14  & none                               &  & 6  & none \\
13  & $\InPD{}[3]=\InPD{}[8]=\trueX{}$   &  & 5  & none   \\
12  & none                               &  & 4  & none               \\
11  & $\InPD{}[6]=\InPD{}[11]=\trueX{}$  &  & 3  & none                  \\
10  & $\InPD{}[2]=\InPD{}[5]=\trueX{}$   &  & 2  & none                  \\
9  & $\InPD{}[1]=\InPD{}[4]=\trueX{}$    &  & 1  & none               \\
8  & none                                &  &    &               \\ \cline{1-2} \cline{4-5}
\end{tabular}
\caption{Algorithm {\small \textbf{TREE PAIRED DOMINATION}} implemented on the tree $T$ in Figure~\ref{fig:root}}
    \label{tab:algo}
\end{table}

\section{Paired domination in random trees}\label{sec:3}

In this section, we are interested in the limiting distribution of paired domination in random trees of order $n$ as $n \to \infty$. The random tree model we have in mind is the so-called \emph{conditioned Galton--Watson tree}, which forms a general model that includes many combinatorial models as special cases.

We assume that all trees are rooted and ordered, i.e., the order of the root branches matters. Given a nonnegative integer-valued random variable $\xi$, the Galton--Watson tree $\tcl$ is a random ordered rooted tree generated by the following process: Begin with a root vertex and create a certain number of children according to $\xi$. Each child then independently generates a number of children according to an independent copy of $\xi$. This process continues for each newly generated vertex until all vertices in the tree have been processed.

Note that in general, such a process may not stop in finite time. Here, we consider only critical Galton--Watson trees where $\mathbb{E}(\xi) = 1$. Furthermore, we assume throughout that
\begin{itemize}
\item $\Pr(\xi = 0) > 0$, and
\item $\gcd \{i \mid \Pr(\xi = i) > 0\} = 1$.
\end{itemize}
Under these conditions, it is well known that $\tcl$ is finite almost surely and that $\Pr(|\tcl| = n) > 0$ for sufficiently large $n$.

The process above defines a probability distribution on the set of finite rooted ordered trees, where, for a given rooted tree $T$,
\[
\Pr(\tcl = T) = \prod_{j \geq 0} \Pr(\xi = j)^{D_j(T)},
\]
and $D_j(T)$ denotes the number of vertices in $T$ with outdegree $j$. This induces a probability distribution on the set of rooted trees of order $n$ by conditioning $\tcl$ to have exactly $n$ vertices. The tree generated under this conditioning is called the conditioned Galton--Watson tree and is denoted by $\tcl_n$. Many models of combinatorial trees can be obtained in this way.

\begin{itemize}
\item {\it Plane trees}: Among all rooted, unlabelled, plane (ordered) trees of order $n$, pick a random tree where each tree is equally likely. This model is equivalent to the conditioned Galton--Watson tree with offspring distribution $\xi \sim \mathrm{Geo}(\tfrac{1}{2})$.

\item {\it Binary trees}: A plane binary tree is a rooted, unlabelled plane tree where each vertex has outdegree at most two. The uniform binary tree model is equivalent to the conditioned Galton--Watson tree with offspring distribution $\xi \sim \mathrm{Bin}(2, \tfrac{1}{2})$.

\item {\it Rooted labelled trees}: Pick one tree uniformly at random among all rooted unordered labelled trees of order $n$. This model also produces a random uniform Cayley tree of order $n$ (by simply ignoring the root) and is equivalent to the conditioned Galton--Watson tree with offspring distribution $\xi \sim \mathrm{Pois}(1)$.

\end{itemize}

We are now ready to state the main result of this section.

\begin{theorem}\label{thm:clt}
Let $\tcl$ be the Galton--Watson tree with offspring distribution $\xi$ such that $\mathbb{E}(\xi) = 1$ and $0 < \mathrm{Var}(\xi) < \infty$. Let $\tcl_n$ denote the corresponding conditioned Galton--Watson tree of order $n$. Then, there exist constants $\mpr > 0$ and $\spr > 0$ such that as $n \to \infty$,
\[
\Exp \gpr(\tcl_n)=\mpr n +o(\sqrt{n})
\
\text{\ and\ }
\
\frac{\gpr(\tcl_n)-\mpr n}{ \spr \sqrt{n}} \dto \mathcal{N}(0,1).
\]
\end{theorem}
Table~\ref{table:mu} shows the numerical values of the constant $\mpr$ for the three combinatorial tree models mentioned above.
\begin{table}[htb]
\centering
\begin{tabular}{|c|c|c|c|c|}
\hline
 & $\xi$ & $\mpr$  \\
\hline
Binary trees & $\mathrm{Bin}(2,1/2)$  & 0.5255  \\
\hline
Plane trees & $\mathrm{Geo}(1/2)$ & 0.4747  \\
\hline
Labelled trees & $\mathrm{Pois}(1)$ & 0.5177 \\
\hline
\end{tabular}
\caption{Numerical values for $\mpr$.}\label{table:mu}
\end{table}

The proof of Theorem~\ref{thm:clt} is carried out in several steps, which make use of the algorithm described in the previous section. When we run the algorithm with an input tree $T$, we label each vertex as $B$, $F$, $R$, or $P$ according to the following rules:
\begin{itemize}
    \item A vertex $i$ is labelled $B$ if, when considering $i$ (in the algorithm),  both $\InPD{}[i]$ and $\PDchild{}[i]$  are false.
    \item A vertex $i$ is labelled $F$ if, when considering $i$, $\InPD{}[i]=\falseX$ and $\PDchild{}[i]=\trueX$.
    \item A vertex $i$ is labelled $R$ if, when considering $i$, $\InPD{}[i]=\trueX$ and none of its children is labelled $R$.
    \item A vertex $i$ is labelled $P$ if, when considering $i$, $\InPD{}[i]=\trueX$ and at least one of its children is labelled $R$.

\end{itemize}
We can also incorporate this vertex-labelling into the algorithm by inserting the following lines between Lines 4 and 5, and also between Lines 14 and 15, of Algorithm TREE PAIRED DOMINATION:
\begin{unnumbered}{} \textbf{} \\
\textbf{Additional Code:}

 \hspace{0.8cm} \textbf{If} {\rm( $\InPD{}[i]=\falseX$ \textbf{and} $\PDchild{}[i]=\falseX$ \rm) }
\item[] \hspace*{2 cm} \textbf{Then}  \{$\ell[i]=B$\}

\hspace{0.8cm} \textbf{If} {\rm( $\InPD{}[i]=\falseX$ \textbf{and} $\PDchild{}[i]=\trueX$ \rm) }
\item[] \hspace*{2 cm} \textbf{Then}  \{$\ell[i]=F$\}

\hspace{0.8cm} \textbf{Elseif} {\rm( \text{for every child} $j$ \text{of} $i$, $\ell(j)\neq R$ \rm) }
\item[] \hspace*{2 cm} \textbf{Then}  \{$\ell[i]=R$\}

\hspace{0.8cm} \textbf{Else}
\item[] \hspace*{2 cm} \textbf{Then}  \{$\ell[i]=P$\}

\end{unnumbered}

Since we have introduced new labels for the vertices, we shall refer to the numerical labels that we considered previously as ranks, e.g., the root has rank $1$. At the end of the modified algorithm, each vertex of $T$ is assigned a unique label: $B$, $F$, $R$, or $P$ (see the example in Figure~\ref{fig:root_label} for illustration). We also remark that no two vertices labelled $R$ are adjacent to each other. The next lemma shows that the label of each vertex can be determined solely by the labels of its children.
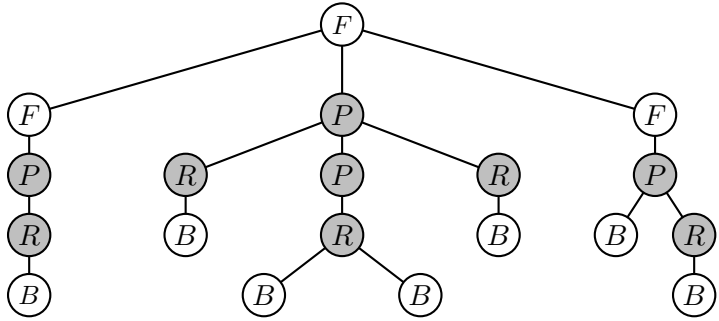
\begin{figure}[htb]
\begin{center}
\begin{tikzpicture}[scale=.8,style=thick,x=1.3cm,y=1cm]
\def\vr{10pt} 
\path (0,1) coordinate (v10);
\path (0,2) coordinate (v5);
\path (0,3) coordinate (v2);
\path (2,1) coordinate (v11);
\path (2,2) coordinate (v6);
\path (3,0) coordinate (v14);
\path (4,1) coordinate (v12);
\path (4,2) coordinate (v7);
\path (4,3) coordinate (v3);
\path (4,3.2) coordinate (v3p);
\path (4,4.5) coordinate (v1);
\path (5,0) coordinate (v15);
\path (6,1) coordinate (v13);
\path (6,2) coordinate (v8);
\path (8,2) coordinate (v9);
\path (8,3) coordinate (v4);
\path (7.5,1) coordinate (v16);
\path (8.5,1) coordinate (v17);
\path (8.5,0) coordinate (v18);
\path (0,0) coordinate (v19);

\draw (v19)--(v10)--(v5)--(v2)--(v1)--(v4)--(v9)--(v16);
\draw (v9)--(v17)--(v18);
\draw (v1)--(v3)--(v7)--(v12)--(v14);
\draw (v12)--(v15);
\draw (v11)--(v6)--(v3)--(v8)--(v13);
\draw (v19) [fill=white] circle (\vr) node  {{\small $B$}};
\draw (v1) [fill=white] circle (\vr) node  {{\small $F$}};
\draw (v2) [fill=white] circle (\vr) node  {$F$};
\draw (v3) [fill=gray!50] circle (\vr) node  {$P$};
\draw (v4) [fill=white] circle (\vr) node  {$F$};
\draw (v5) [fill=gray!50] circle (\vr) node  {$P$};
\draw (v6) [fill=gray!50] circle (\vr) node  {$R$};
\draw (v7) [fill=gray!50] circle (\vr) node  {$P$};
\draw (v8) [fill=gray!50] circle (\vr) node  {$R$};
\draw (v9) [fill=gray!50] circle (\vr) node  {$P$};
\draw (v10) [fill=gray!50] circle (\vr)  node  {$R$};
\draw (v11) [fill=white] circle (\vr) node  {$B$};
\draw (v12) [fill=gray!50] circle (\vr) node  {$R$};
\draw (v13) [fill=white] circle (\vr) node  {$B$};
\draw (v14) [fill=white] circle (\vr)  node  {$B$};
\draw (v15) [fill=white] circle (\vr) node  {$B$};
\draw (v16) [fill=white] circle (\vr) node  {$B$};
\draw (v17) [fill=gray!50] circle (\vr) node  {$R$};
\draw (v18) [fill=white] circle (\vr) node  {$B$};

\end{tikzpicture}
\end{center}
\caption{A tree with labelled vertices}
\label{fig:root_label}
\end{figure}
\begin{lemma}\label{lem:pd-additive1}
    If the vertices of a tree $T$ are labelled according to the rules above, then for every $v$ in $T$ we have
    \begin{itemize}
        \item $v$ is labelled $B$ if and only if $v$ is a leaf or every child of $v$ is labelled $F$,
        \item $v$ is labelled $F$ if and only if $v$ has at least one child labelled $P$ and all the other children of $v$ are labelled $F$,

        \item  $v$ is labelled $R$ if and only if $v$ has at least one child labelled $B$ and all the other children of $v$ are labelled $P$ or $F$,
        \item  $v$ is labelled $P$ if and only if $v$ has at least a child labelled $R$.
    \end{itemize}
\end{lemma}

\begin{proof}
    For a given rooted tree $T$, we label the vertices of $T$ in two ways: according to the original rules and according to the rules in the lemma. Assume, by contradiction, that these give two different labellings of $T$. Let $v$ be the vertex with the highest rank (farthest from the root) such that the labelling given by the original definition differs from the labelling given in the statement of the lemma. Note that $v$ cannot be a leaf, since leaves are labelled $B$ in both labellings. Thus, we may assume from now on that $v$ is an interior vertex. There are then different cases to consider:

    \begin{itemize}
    \item Assume that the statement of the lemma gives $v$ the label $B$. This means that all children of $v$ are labelled $F$. Since both statements agree on the labelling of the children of $v$, the algorithm would have passed through each of these children without changing any of the vertex states. Additionally, no grandchildren of $v$ are labelled $B$, since this would change the $\InPD$ state of a child of $v$ to true. Thus, by the time $v$ is considered by the algorithm, both $\PDchild{}[v]$ and $\InPD{}[v]$ remain false. Therefore, the definition would also give $v$ the label $B$.
    \item Assume that the statement of the lemma gives $v$ the label $F$. This means that at least one of its children is labelled $P$, and the rest are labelled $F$ (by both rules). This implies that when $v$ is considered by the algorithm, $\PDchild{}[v] = \trueX$. It remains to show that $\InPD{}[v]$ is false. Initially, $\InPD{}[v]$ is set to false. If $\InPD{}[v]$ has been changed to true before $v$ has been considered by the algorithm, then $v$ must have a grandchild labelled $B$. Let $w$ be such a $B$-grandchild of $v$ with the highest rank. Then, the parent of $w$ will be labelled $R$ (by both rules). This is a contradiction since no child of $v$ is labelled~$R$.
    \item Assume that the statement of the lemma gives $v$ the label $R$. Then, $v$ has some children labelled $B$ and the rest are labelled $P$ or $F$ (by both rules). It is clear that no child is labelled $R$. Let $w$ be the $B$-child of $v$ with the highest rank. By the time the algorithm considers $w$, both $v$ and its parent have $\InPD$ and $\PDchild$ set to false. So, after considering $w$, $\InPD{}[v] = \trueX$. Hence, $v$ will also be given the label $R$ by the definition.
    \item Finally, assume that the statement of the lemma gives $v$ the label $P$. Then, $v$ must have a child labelled $R$, which in turn must have a child labelled $B$. Let $w_1$ be the $R$-child of $v$ with the highest rank, and let $w_2$ be the $B$-child of $w_1$ with the highest rank. When $w_2$ is considered by the algorithm, both $\InPD{}[w_1]$ and $\InPD{}[v]$ are set to true. Therefore, $v$ is also labelled $P$ by the definition.
\end{itemize}
Hence, both the original rules and the rules in the lemma give the same label to $v$, which is a contradiction.
\end{proof}

It is useful to partition the class of rooted plane trees $\mathbb{T}$ into four subclasses: $\mathbb{T}_B$, $\mathbb{T}_F$, $\mathbb{T}_R$, and $\mathbb{T}_P$, according to the label of the root. For instance, a tree in $\mathbb{T}_B$ has its root vertex labelled $B$. Let $\Phi(T)$ be the number of vertices labelled $R$ in $T$ after running the algorithm with the vertex labelling described above. By convention, we assume that a tree consisting of a single vertex belongs to $\mathbb{T}_B$; hence, $\Phi(\bullet) = 0$.

\begin{lemma}\label{lem:pd-additive}
For every tree $T$ of order at least $2$, we have
\begin{equation}\label{eq:additivity}
\Phi(T)=\sum_{k=1}^{d}\Phi(T_k)+\varphi(T)
\end{equation}
where $T_1, T_2, \ldots,T_d$ are the root-branches of $T$, and
\[
\varphi(T)=
\begin{cases}
1 \ \text{ if } \ T\in \mathbb{T}_R\\
0 \ \text{ otherwise}.
\end{cases}
\]
Moreover, we have
\[
\gpr(T)=2\Phi(T)+
\begin{cases}
2 \ \text{ if } \ T\in \mathbb{T}_B\\
0 \ \text{ otherwise}.
\end{cases}
\]
\end{lemma}
\begin{proof}
Since $\Phi(T)$ counts the number of $R$-vertices in $T$ after running the algorithm, the first part of the lemma is obvious, as we only need to sum the number of $R$-vertices from all the branches and add the root of $T$ if it is labelled $R$.

The second part of the lemma is proved in a few steps. Let us first prove the following lower bound:
\begin{equation}\label{eq:lower}
\gpr(T) \geq 2\Phi(T).
\end{equation}

To see this, notice that the $\gpr$-set produced by the algorithm contains all the $R$-vertices. In addition, the $R$-set is independent by definition of the label $R$. Each of these $R$-vertices is paired with another vertex that is not labelled $R$ in the $\gpr$-set, so the claim is proved.

Next, let us denote by $S$ the set of vertices labelled $R$ or $P$. We claim that $S$ dominates $T$, except possibly the root. To see this, assume that there is a vertex $v$ different from the root that is not dominated by $S$. This implies that $\ell(v) \notin \{R, P\}$ and no neighbor of $v$ is labelled $P$ or $R$. In particular, the labels of the children of $v$ are either $B$ or $F$. Thus, we have the following possible cases:
\begin{itemize}
    \item If $v$ has a $B$-child, then $\InPD[v] = \trueX$. So, by definition, $\ell(v)$ must either be $R$ or $P$, which is a contradiction.
    \item If all children of $v$ are labelled $F$, then $\ell(v) = B$. By Lemma~\ref{lem:pd-additive1}, the parent of $v$ cannot be labelled $B$ or $F$. Therefore, the parent of $v$ has to be labelled $P$ or $R$, which is again a contradiction.
\end{itemize}

We can then conclude that the set $S$ of $R$ and $P$ vertices dominates $T$, except possibly the root. Let us now construct a PD-set from $S$. Consider the set $S^*$ in the following way: we go from the top to the bottom of $T$, and at each $P$-vertex, we pair it with one of its $R$-children, which we know exists by Lemma~\ref{lem:pd-additive1}. Then, for the remaining $R$-vertices which have not yet been paired, pair each with one of its $B$-children (which we also know exists), and add this selected $B$-child to $S^*$.

We already know that $S^*$ dominates $T$, except possibly the root, which we will denote by $r$. So, we must separately consider the cases where the root is dominated or not. Observe first that the root is not dominated by $S^*$ if and only if the root is labelled $B$. To see this, if the root $r$ is not dominated by $S^*$, then none of its children is labelled $R$ or $P$. Moreover, the root cannot have a child labelled $B$ either, since otherwise, it would be labelled $R$ according to Lemma~\ref{lem:pd-additive1}. This means that all the children are labelled $F$, which implies that $\ell(r) = B$. Conversely, if $\ell(r) = B$, then all its children are labelled $F$ by Lemma~\ref{lem:pd-additive1}, so they cannot be in $S^*$. Thus, $r$ is not dominated by $S^*$. Now consider the following two cases:
\begin{itemize}
    \item If the root $r$ is dominated by $S^*$ (equivalently, $\ell(r) \neq B$), then $S^*$ is a PD-set of $T$. Therefore,
    \[
    \gpr(T) \leq |S^*| = 2\Phi(T).
    \]
    This is in fact an equality by Equation~\eqref{eq:lower}. This proves the case $T \notin \mathbb{T}_B$.
    \item If the root is not dominated by $S^*$ (equivalently, $\ell(r) = B$), then clearly none of the children of the root is in $S^*$. Therefore, we can add the root and one of its children to $S^*$ to obtain a PD-set of $T$. We obtain
    \[
    \gpr(T) \leq 2\Phi(T) + 2.
    \]
    We claim that we also have
    \[
    \gpr(T) \geq 2\Phi(T) + 2.
    \]
Indeed, suppose to the contrary that this is not the case, i.e., $\gpr(T) = 2\Phi(T)$. Then, let us take the $\gpr$-set from the algorithm, with an appropriate pairing of the elements. In this case, since the $R$-vertices are independent, each pair of this $\gpr$-set must have one $R$-vertex. Now, since it is a dominating set, it must contain a child of the root, say $w$. This means that $\ell(w) = R$ or $w$ has a child labelled $R$. In either case, $\ell(w)$ cannot be $F$, which is in contradiction with $\ell(r) = B$ (according to Lemma\ref{lem:pd-additive1}).
\end{itemize}
This completes the proof of the lemma.
\end{proof}

Tree parameters that satisfy a recursion like the one in Equation~\eqref{eq:additivity} are called additive. In our case, the function $\varphi$ is referred to as the associated toll function of the additive parameter~$\Phi$. There are many results in the literature concerning the limiting distributions of additive parameters of random rooted trees; see, for example, \cite{ janson2016asymptotic, Naina, wagner2015central} and the references therein. The result relevant to our case can be found in \cite[Theorem 1]{Naina}. For completeness, we will state this theorem in full. But before we do so, let us recall some background and notation.

We consider $F$ to be an additive tree functional with a toll function $f$, which means that for every rooted tree $T$, we have
\[
F(T)=\sum_{k=1}^dF(T_k)+f(T).
\]
where $T_1, T_2, \ldots,T_d$ are the root-branches of $T$. We also assume that the toll function $f$ is bounded.

For a tree $T$ and a positive integer $M$, let $T^{(M)}$ denote the tree consisting of all vertices whose distance to the root is at most $M$. According to Kesten \cite{kesten1986subdiffusive}, we define the infinite size-biased Galton--Watson tree $\hat \tcl$ in the following way: Starting from an infinite path beginning at the root, we add to each vertex on the path and each side of the path a random number of additional branches, each of which is an independent copy of $\tcl$.  The random infinite tree $\hat \tcl$ plays an important role in the study of conditioned Galton--Watson trees because it is known that the truncated random tree $\tcl_n^{(M)}$ converges in distribution to $\hat \tcl^{(M)}$ for every fixed $M$, more details on this can be found in \cite{J12}. The distribution of $\hat \tcl^{(M)}$ is given by the following formula:
\[
\Pr(\hat \tcl^{(M)} = T) = w_M(T) \Pr(\tcl^{(M)} = T),
\]
where $w_M(T)$ denotes the number of vertices at distance $M$ from the root in $T$, see~\cite[Equation~(5.11)]{J12}. The above equation defines a probability distribution if $\mathbb{E}(\xi) = 1$ since in this case we can show that $\mathbb{E}(w_M(\tcl))=1$.


\begin{theorem}[\cite{Naina}]\label{thm:normality}
	Let $\tcl_n$ be a conditioned Galton--Watson tree of order $n$ with offspring distribution $\xi$, where $\xi$ satisfies $\Exp (\xi)= 1$ and $0<\sigma^2:=\mathrm{Var} (\xi) <\infty$. Consider an additive parameter $F$ with a bounded toll function $f$.  Furthermore, assume that there exists a sequence $(p_M)_{M\geq 1}$ of positive numbers with  $p_M\to 0$ as $M\to\infty$ such that
	
	\begin{itemize}
		\item for every $M, N \in \{1, 2, \dots\}$, such that $N \ge M$
		\begin{equation}\label{eq:thm1-2}
		 	\, \Exp \left|  f(\hat \tcl^{(M)}) -\Exp\left(f(\hat{\tcl}^{(N)})\,  | \, \hat{\mathcal{T}}^{(M)}\right)\right|\leq p_M,
		\end{equation}
		and
		\item there is a sequence of positive integers $(M_n)_{n\geq 1}$ such that for large enough $n$,
		\begin{equation}\label{eq:thm1-3}
		\Exp \left|f(\mathcal{T}_n)-f\left(\mathcal{T}_n ^{(M_n)}\right)\right|\leq p_{M_n}.
		\end{equation}
	\end{itemize}
	If $a_n:=n^{-1/2}(np_{M_n}+M_n^2)$ satisfies
	\begin{equation}\label{eq:thm1-4}
	\lim_{n\to\infty} a_n= 0, \, \text{ and } \, \sum_{n=1}^{\infty}\frac{a_n}{n}<\infty,
	\end{equation}
	then
	\begin{equation}\label{eq:normality}
	\frac{F(\mathcal{T}_n) - \mu n}{\sqrt n} \dto \mathcal{N}(0,\gamma^2)
	\end{equation}
	where $\mu=\Exp f(\mathcal{T})$, for some $0 \le \gamma < \infty$. Moreover, we have
 \begin{equation}
  \Exp ( F(\tcl_n) ) =\mu n+o(\sqrt{n} \,), \text{\ and\ } \mathrm{Var} F(\tcl_n) =\gamma^2n+o(n).
 \end{equation}
\end{theorem}

In order to prove Theorem~\ref{thm:clt}, we need to show that the toll function $\varphi$ associated with the functional $\Phi$ defined in Lemma~\ref{lem:pd-additive} satisfies the conditions of Theorem~\ref{thm:normality}. This is our main goal in the proof. The conditions of Theorem~\ref{thm:normality} essentially say that there are not ``too many" trees  $T$ such that $f(T)\neq f(T^{(M)})$. For our cases, this means that if we run our algorithm (twice), once with input $T$ and once with $T^{(M)}$, then most of the time the label of the root remains unchanged. To make this precise, we consider the following set of trees for $M> 3$:
\[
\mathcal{B}_{M}=\left\lbrace T \ \text{possibly infinite}\ : \ \exists N>M,  \ \ell\Big(\rm{root}(T^{(N)})\Big)\neq \ell\left(\rm{root}(T^{(M)})\right) \right\rbrace,
\]
where $\ell(\rm{root}(T))$ indicates the label of the root of $T$, i.e.,  $B$, $F$, $R$ or $P$, after running our algorithm with input $T$.  Observe that there exists a tree $\tau_0$ of height $3$ with the property that $\Pr(\mathcal{T}=\tau_0)>0$ and that for every finite $T$ such that $T^{(3)}=\tau_0$,
\[
\ell(\rm{root}(T))= \ell\left(\rm{root}(T^{(3)})\right).
\]
The existence of a tree $\tau_0$ with the above property is crucial in the proof of Theorem~\ref{thm:clt}, so let us explicitly construct an example of such a tree. We define
\begin{equation}\label{eq:d0}
d_0:=\min \{d\geq 2 :\ \Pr(\xi =d)\neq 0 \},
\end{equation}
whose existence follows from our assumptions on $\xi$. Then, the tree $\tau_0$ described in Figure~\ref{fig:tau}, where each of the four internal vertices has outdegree $d_0$, satisfies the desired properties. It is easy to see that if we run the algorithm on any extension of $\tau_0$, then the root will always be labelled $P$ since the root of the left-most branch will be labelled~$R$.
\begin{figure}[htb]
\begin{center}
\begin{tikzpicture}[scale=.8,style=thick,x=1cm,y=1cm]
\def\vr{2.5pt} 
\path (4,4.5) coordinate (v1);
\path (0,3) coordinate (v2);
\path (2.5,3) coordinate (v3);
\path (5.5,3) coordinate (v4);
\path (8,3) coordinate (v5);
\path (-1,1.5) coordinate (v6);
\path (1,1.5) coordinate (v7);
\path (7,1.5) coordinate (v8);
\path (9,1.5) coordinate (v9);
\path (8,0) coordinate (v10);
\path (10,0) coordinate (v11);
\draw (v1)--(v2);
\draw (v1)--(v3);
\draw (v1)--(v4);
\draw (v1)--(v5);
\draw (v2)--(v6);
\draw (v2)--(v7);
\draw (v5)--(v8);
\draw (v5)--(v9);
\draw (v10)--(v9);
\draw (v11)--(v9);
\draw [decorate,decoration={brace,amplitude=6pt},xshift=0pt,yshift=0pt]
    (1.1,1.1) -- (-1.2,1.1) node [black,midway,xshift=0cm,yshift=-0.5cm]
    {$d_0$};
\draw [decorate,decoration={brace,amplitude=6pt},xshift=0pt,yshift=0pt]
    (5.7,2.6) -- (2.3,2.6) node [black,midway,xshift=0cm,yshift=-0.5cm]
    {$d_0-2$};

%
%
\draw (v1) [fill=black] circle (\vr);
\draw (v2) [fill=black] circle (\vr);
\draw (v3) [fill=white] circle (\vr);
\draw (v4) [fill=white] circle (\vr);
\draw (v5) [fill=black] circle (\vr);
\draw (v6) [fill=white] circle (\vr);
\draw (v7) [fill=white] circle (\vr);
\draw (v8) [fill=white] circle (\vr);
\draw (v9) [fill=black] circle (\vr);
 \draw (v10) [fill=white] circle (\vr);
 \draw (v11) [fill=white] circle (\vr);
 \draw[anchor = south] (0,1.2) node {{ $\cdots$}};
 \draw[anchor = south] (4,2.7) node {{ $\cdots$}};
 \draw[anchor = south] (8,1.2) node {{ $\cdots$}};
 \draw[anchor = south] (9,-.3) node {{ $\cdots$}};
\end{tikzpicture}
\end{center}
\caption{Fixed tree $\tau_0$ of height $3$.}\label{fig:tau}
\end{figure}
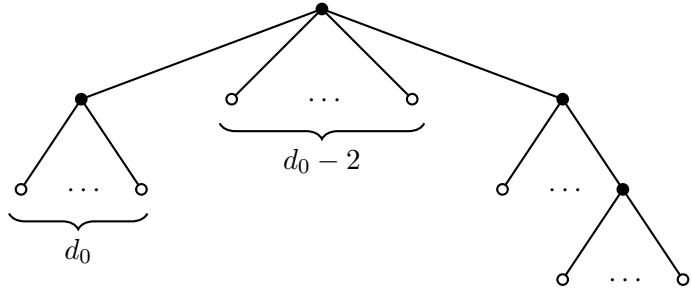

\begin{lemma}\label{lem:t-bound}
There exists a constant $c_1\in (0,1)$ such that
$\Pr(\mathcal{T}\in \mathcal{B}_{M}) = O (c_1^{M})$ as $M\to\infty$.
\end{lemma}
\begin{proof}
Observe, by definition of $\tau_0$, that for any tree $T$ to be in $\mathcal{B}_{M}$, we must have $T^{(3)}\neq \tau_0$. Hence, we can write
\[
\Pr(\mathcal{T}\in \mathcal{B}_{M}) =\sum_{T\neq \tau_0} \Pr(\mathcal{T}^{(3)}= T) \Pr(\mathcal{T}\in \mathcal{B}_{M} | \mathcal T^{(3)}=T ).
\]
For any tree $T$, recall that $w_3(T)$ denotes the number of vertices at distance 3 from the root. From the definition of $\mathcal{T}$, given the event $\{\mathcal{T}^{(3)}= T\}$, the rest of $\mathcal{T}$ consists of $w_3(T)$ independent copies of $\mathcal{T}$. On the other hand, if we also have $\mathcal{T}\in \mathcal{B}_{M}$, then at least one of those $w_3(T)$ branches of $\tcl$ is an element of $\mathcal{B}_{M-3}$ (if $M>3$). Thus, we have
\[
\Pr(\mathcal{T}\in \mathcal{B}_{M} | \mathcal T^{(3)}=T )\leq w_3(T) \Pr(\mathcal{T}\in \mathcal{B}_{M-3}).
\]
Therefore,
\begin{align*}
\Pr(\mathcal{T}\in \mathcal{B}_{M}) & \leq  \Pr(\mathcal{T}\in \mathcal{B}_{M-3})\sum_{T\neq \tau_0} \Pr(\mathcal{T}^{(3)}= T)w_3(T)\\
& =\Pr(\mathcal{T}\in \mathcal{B}_{M-3}) \Big(\mathbb{E}(w_3(\mathcal{T}))-w_3(\tau_0)\Pr(\mathcal{T}=\tau_0)\Big).
\end{align*}
It is clear that the term in brackets in the last expression is independent of $M$ and is nonnegative. Moreover, since we know that $\mathbb{E}(w_3(\mathcal{T}))=1$ and $w_3(\tau_0)\Pr(\mathcal{T}=\tau_0)>0$, this term is also strictly less than 1. Thus, the lemma is proved by iteration.
\end{proof}

\begin{lemma}\label{lem:ht-bound}
There exists a constant $c_2\in (0,1)$ such that
$\Pr(\hat \tcl\in \mathcal{B}_{M})=O(c_2^{M})$ as $M\to\infty$.
\end{lemma}
\begin{proof}
The proof follows the same lines as the proof of the previous lemma. We have
\[
\Pr(\hat \tcl \in \mathcal{B}_{M}) =\sum_{T\neq \tau_0} \Pr(\hat \tcl ^{(3)}= T) \Pr(\hat \tcl \in \mathcal{B}_{M} | \hat \tcl^{(3)}=T ).
\]
Given the event $\{\hat{\tcl}^{(3)} = T\}$, the rest of $\hat{\tcl}$ consists of $w_3(T)$ independent random trees, where one is a copy of $\hat{\tcl}$ and the rest are copies of $\mathcal{T}$. Note that these branches are ordered, and the position of the copy of $\hat{\tcl}$ in this order is uniform. Hence, we have
\[
\Pr(\hat \tcl \in \mathcal{B}_{M} | \hat \tcl^{(3)}=T )\leq \Pr(\hat \tcl \in \mathcal{B}_{M-3})+(w_3(T)-1)\Pr( \tcl \in \mathcal{B}_{M-3}).
\]
Thus, we deduce that
\[
\Pr(\hat \tcl \in \mathcal{B}_{M})\leq \left(\sum_{T\neq \tau_0} \Pr(\hat \tcl ^{(3)}= T)\right)\Pr(\hat \tcl \in \mathcal{B}_{M-3})+\mathbb{E}(w_3(\hat \tcl)-1) \Pr( \tcl \in \mathcal{B}_{M-3}).
\]
We know that $\mathbb{E}(w_3(\hat{\tcl}) - 1)$ is finite (see for example \cite[Lemma 2.3]{janson2006random}); hence, by Lemma~\ref{lem:t-bound}, the second term on the right-hand side is bounded by $O(c_1^M)$. On the other hand,
\[
0\leq \sum_{T\neq \tau_0} \Pr(\hat \tcl ^{(3)}= T)=1- \Pr(\hat \tcl ^{(3)}= \tau_0)<1.
\]
Once again, the lemma is proved by iteration.
\end{proof}

We are now ready to prove our main theorem.

\begin{proof}[Proof of Theorem~\ref{thm:clt}]
Observe that if $T\notin \mathcal{B}_{M}$, then for every $N\geq M$,
$\varphi(T^{(N)})=\varphi(T^{(M)}).$
So, since the value of $\varphi$ can only be $0$ or $1$, we have
\[
\left|  \varphi(\hat \tcl^{(M)}) -\Exp\left(\varphi(\hat{\tcl}^{(N)})\,  | \, \hat{\mathcal{T}}^{(M)}\right)\right|\leq \mathbbm{1}_{\{\hat\tcl \in \mathcal{B}_M\}}.
\]
Using Lemma~\ref{lem:ht-bound}, we obtain
\begin{equation}\label{eq:local_hat}
\mathbb{E}  \left|  \varphi(\hat \tcl^{(M)})  -\Exp\left(\varphi(\hat{\tcl}^{(N)})\,  | \, \hat{\mathcal{T}}^{(M)}\right)\right|\leq \Pr(\hat\tcl \in \mathcal{B}_M)\ll c_2^M.
\end{equation}

By the same argument, we also have
\begin{equation*}
\mathbb{E} \left|  \varphi(\tcl_{n})- \varphi(\tcl_{n}^{(M)})\right|\leq \Pr(\tcl_n \in \mathcal{B}_M)\leq \frac{\Pr(\tcl \in \mathcal{B}_M)}{\Pr(|\tcl|=n)}.
\end{equation*}
Based on our assumptions on $\xi$, it is well known that $\Pr(|\tcl| = n) = \Theta(n^{-3/2})$; see Janson~\cite[Equation~(4.13)]{janson2016asymptotic}, which also refers to Kolchin~\cite{Kolchin1984}.
So, by Lemma~\ref{lem:t-bound}, we have
\begin{equation}\label{eq:localn}
\mathbb{E} \left|  \varphi(\tcl_{n})- \varphi(\tcl_{n}^{(M)})\right|=O(n^{3/2}c_1^{M}).
\end{equation}
If we choose $M_n := (\log n)^2$ and $p_M := c_3^M$, where $\max\{c_1, c_2\} < c_3 < 1$, then it is not difficult to verify that all conditions of Theorem~\ref{thm:normality} are satisfied. Therefore, we obtain the asymptotic normality of $\Phi(\tcl_n)$, with the estimates for its mean and variance exactly as stated in Theorem~\ref{thm:normality}. Furthermore, the same result follows for $\gpr(\tcl_n)$, since by Lemma~\ref{lem:pd-additive}, we have
\[
|2\Phi(\tcl_n) - \gpr(\tcl_n)| \leq 2.
\]
Hence, the constant term  $\mpr$ in Theorem~\ref{thm:clt} is equal to $2\mathbb{E}(\varphi(\tcl))$. On the other hand, we also have
\[
\mathrm{Var}(\gpr(\tcl_n))=\spr^2n+o(n).
\]

It remains to show that the constant $\spr^2$ is nonzero, to guarantee a nondegenerate Gaussian limit law for $\gpr(\tcl_n)$.  This follows from an argument already used in the proof of \cite[Theorem 10]{Naina}. By the same reasoning, it is enough to show that there exist two finite trees, $\tau_1$ and $\tau_2$, of the same size which satisfy the following properties:
\begin{itemize}
\item $\Pr(\tcl = \tau_i) > 0$ for $i \in \{1,2\}$,
\item $\gpr(\tau_1)\neq \gpr(\tau_2)$, and 
\item $\ell(\rm{root}(\tau_1))=\ell(\rm{root}(\tau_2)).$
\end{itemize}

To briefly explain the idea, if we condition on the structure obtained by deleting every occurrence of $\tau_1$ or $\tau_2$ as a fringe subtree of $\tcl_n$, then $\gpr(\tcl_n)$ is equal to a fixed term plus a constant (equal to $|\gpr(\tau_1) - \gpr(\tau_2)|$) times a sum of independent, nondegenerate Bernoulli variables, where the number of these Bernoulli variables is exactly the number of occurrences of $\tau_1$ and $\tau_2$ in $\tcl_n$. Since, on average, the number of such occurrences in $\tcl_n$ is $\gg n$ (see, for example, \cite[Theorem 1.3]{janson2016asymptotic} which refers to \cite{aldous1991} among others), the law of total variance implies that $\mathrm{Var}(\gpr(\tcl_n)) \gg n$.

Now, to construct such trees $\tau_1$ and $\tau_2$, let $d_0$ be as defined in Equation~\eqref{eq:d0}. Then the two full $d_0$-ary trees shown in Figure~\ref{fig:tau12} satisfy the second property. We denote these trees, from left to right in Figure~\ref{fig:tau12}, by $t_1$ and $t_2$, where the dots indicate leaves that need to be added to complete the trees as full $d_0$-ary trees. It is clear that the paired domination numbers of $t_1$ and $t_2$ are 4 and 6, respectively, with the $\gpr$-set consisting of the black-coloured vertices in the figure.
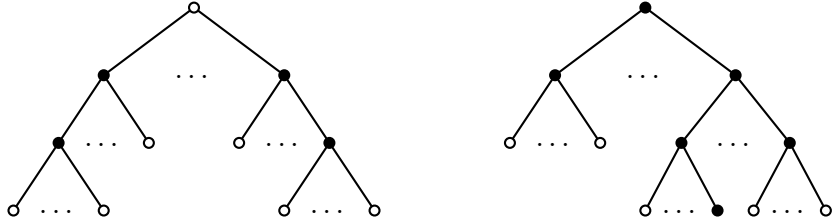
\begin{figure}[htb]
\begin{center}
\begin{tikzpicture}[scale=.75,style=thick,x=.8cm,y=.8cm]
\def\vr{2.5pt} 

\path (0,4.5) coordinate (v1a);
\path (-2,3) coordinate (v2a);
\path (2,3) coordinate (v3a);
\path (-3,1.5) coordinate (v4a);
\path (-1,1.5) coordinate (v5a);
\path (1,1.5) coordinate (v6a);
\path (3,1.5) coordinate (v7a);
\path (-4,0) coordinate (v8a);
\path (-2,0) coordinate (v9a);
\path (2,0) coordinate (v10a);
\path (4,0) coordinate (v11a);

\path (10,4.5) coordinate (v1b);
\path (8,3) coordinate (v2b);
\path (12,3) coordinate (v3b);
\path (10.8,1.5) coordinate (v6b); 
\path (9,1.5) coordinate (v5b);
\path (7,1.5) coordinate (v4b);
\path (13.2,1.5) coordinate (v7b);
\path (10,0) coordinate (v8b);
\path (11.6,0) coordinate (v9b);
\path (12.4,0) coordinate (v10b);
\path (14,0) coordinate (v11b);

\draw (v1a)--(v2a);
\draw (v1a)--(v3a);
\draw (v2a)--(v4a);
\draw (v2a)--(v5a);
\draw (v3a)--(v6a);
\draw (v3a)--(v7a);
\draw (v4a)--(v8a);
\draw (v4a)--(v9a);
\draw (v7a)--(v10a);
\draw (v7a)--(v11a);

\draw (v1b)--(v2b);
\draw (v1b)--(v3b);
\draw (v2b)--(v5b);
\draw (v2b)--(v4b);
\draw (v3b)--(v6b);
\draw (v3b)--(v7b);

\draw (v6b)--(v8b); 
\draw (v6b)--(v9b);
\draw (v7b)--(v10b);
\draw (v7b)--(v11b);

\foreach \v in {2a,3a,4a,7a} \draw (v\v) [fill=black] circle (\vr);
\foreach \v in {1a,5a,6a,8a,9a,10a,11a} \draw (v\v) [fill=white] circle (\vr);

\foreach \v in {1b,2b,3b,7b,6b,9b} \draw (v\v) [fill=black] circle (\vr);
\foreach \v in {4b,5b,8b,10b,11b} \draw (v\v) [fill=white] circle (\vr);

\draw[anchor=south] (0,2.7) node {\dots};
\draw[anchor=south] (-2,1.2) node {\dots};
\draw[anchor=south] (2,1.2) node {\dots};
\draw[anchor=south] (-3,-.3) node {\dots};
\draw[anchor=south] (3,-.3) node {\dots};

\draw[anchor=south] (10,2.7) node {\dots};
\draw[anchor=south] (8,1.2) node {\dots};
\draw[anchor=south] (12,1.2) node {\dots};
\draw[anchor=south] (10.8,-.3) node {\dots};
\draw[anchor=south] (13.2,-.3) node {\dots};

\end{tikzpicture}
\end{center}
\caption{Two trees, $t_1$ and $t_2$, of the same size with different paired domination numbers.}\label{fig:tau12}
\end{figure}

Finally, we obtain the trees $\tau_1$ and $\tau_2$ by attaching $t_1$ and $t_2$ (respectively) to the rightmost leaf at level 3 of the tree $\tau_0$ shown in Figure~\ref{fig:tau}. It is straightforward to verify that the trees $\tau_1$ and $\tau_2$ constructed in this way satisfy all the desired properties. This completes the proof. \end{proof}

We conclude this section with a method for computing the constant $\mpr$ and some further discussion. From the main theorem, we have $\mpr = 2\mathbb{E}(\varphi(\tcl))$, and since $\varphi$ is an indicator function for $\mathbb{T}_R$, we have
\begin{align}
    \mpr=2\Pr(\tcl\in \mathbb{T}_R).
\end{align}
It is convenient to define $x_X:=\Pr(\tcl\in \mathbb{T}_X)$ for each $X\in \{B,F,R,P\}$. Then, for each $X$, we can write
\[
x_X =\sum_{k=0}^{\infty}\Pr(\xi=k)\Pr(\tcl\in \mathbb{T}_X | \deg \tcl =k).
\]
Applying the characterisations in Lemma~\ref{lem:pd-additive1} , we immediately deduce that
\begin{align*}
\Pr(\tcl\in \mathbb{T}_B | \deg \tcl =k) & = \Pr(\tcl\in \mathbb{T}_F)^k,\\
\Pr(\tcl\in \mathbb{T}_F | \deg \tcl =k) & = \left(\Pr(\tcl\in \mathbb{T}_P)+\Pr(\tcl\in \mathbb{T}_F)\right)^k-\Pr(\tcl\in \mathbb{T}_F)^k,\\
\Pr(\tcl\in \mathbb{T}_R | \deg \tcl =k) & = \left(\Pr(\tcl\in \mathbb{T}_B)+\Pr(\tcl\in \mathbb{T}_P)+\Pr(\tcl\in \mathbb{T}_F)\right)^k-\left(\Pr(\tcl\in \mathbb{T}_P)+\Pr(\tcl\in \mathbb{T}_F)\right)^k, \text{ and }\\
\Pr(\tcl\in \mathbb{T}_P | \deg \tcl =k) & = 1-\left(\Pr(\tcl\in \mathbb{T}_B)+\Pr(\tcl\in \mathbb{T}_P)+\Pr(\tcl\in \mathbb{T}_F)\right)^k.
\end{align*}
Let us denote by $g$ the probability generating function of $\xi$, i.e.,
\[
g(x):=\sum_{k=0}^{\infty}\Pr(\xi=k)x^k,
\]
which converges for every $x$ in the interval $[0,1]$ with $g(1)=1$. Then, from the above equations, the constants $x_B$, $x_F$, $x_R$, and $x_P$ satisfy the following system of equations:
\begin{equation}\label{eq:syst}
\begin{cases}
x_B & =g(x_F)\\
x_F & =g(x_F+x_P)-g(x_F)\\
x_R & =g(x_B+x_F+x_P)-g(x_F+x_P)\\
x_P & = 1- g(x_B+x_F+x_P).
\end{cases}
\end{equation}
Since $\mathrm{Var}(\xi) < \infty$ implies that $g(x)$ is at least twice differentiable on the interval $(0, 1)$, and also that $g'(x) = \mathbb{E}(\xi) = 1$, we can use these facts to show that the above system of equations has a unique solution (e.g., by showing that the Jacobian determinant is nonzero). Then, such a solution can be numerically computed for specific examples, such as binary trees, plane trees, and labelled rooted trees, as shown in Table~\ref{tab:numeric}.
\begin{table}[htb]
\centering
\begin{tabular}{|c|c|c|c|c|c|c|c|}
\hline
 & $g(x)$ & $x_B$ &$x_F$ &$x_R$ & $x_P$\\
\hline
Binary trees & $(1+x)^2/4$  & 0.3347  & 0.1571 & 0.2627 & 0.2455\\
\hline
Plane trees & $(2-x)^{-1}$ & 0.5145 & 0.0563 & 0.2374 & 0.1918\\
\hline
Labelled trees & $\exp(x-1)$ & 0.4085 & 0.1046 & 0.2589 & 0.2281\\
\hline
\end{tabular}
\caption{Numerical solutions of Equation~\eqref{eq:syst}.}\label{tab:numeric}
\end{table}

In the case of labelled trees, each Cayley tree on the vertex set $\{1, 2, \dots, n\}$ corresponds to $n$ distinct labelled rooted trees, determined by the choice of the root. Therefore, choosing a random Cayley tree on the vertex set $\{1, 2, \dots, n\}$ uniformly is equivalent to selecting a labelled rooted tree and then forgetting the root. As a result, the distribution of the paired domination number of a random Cayley tree of $n$ is also asymptotically normal, with mean and variance as stated in Theorem~\ref{thm:clt}.

Figure~\ref{fig:histo} shows a histogram of paired domination numbers of 20000 randomly (and independently) sampled Cayley trees of size 10000. The simulation was run using SageMath~\cite{sagemath}. The horizontal axis represents the actual value of the paired domination number.
\begin{figure}[htb]
    \centering
    \includegraphics[width=0.6\linewidth]{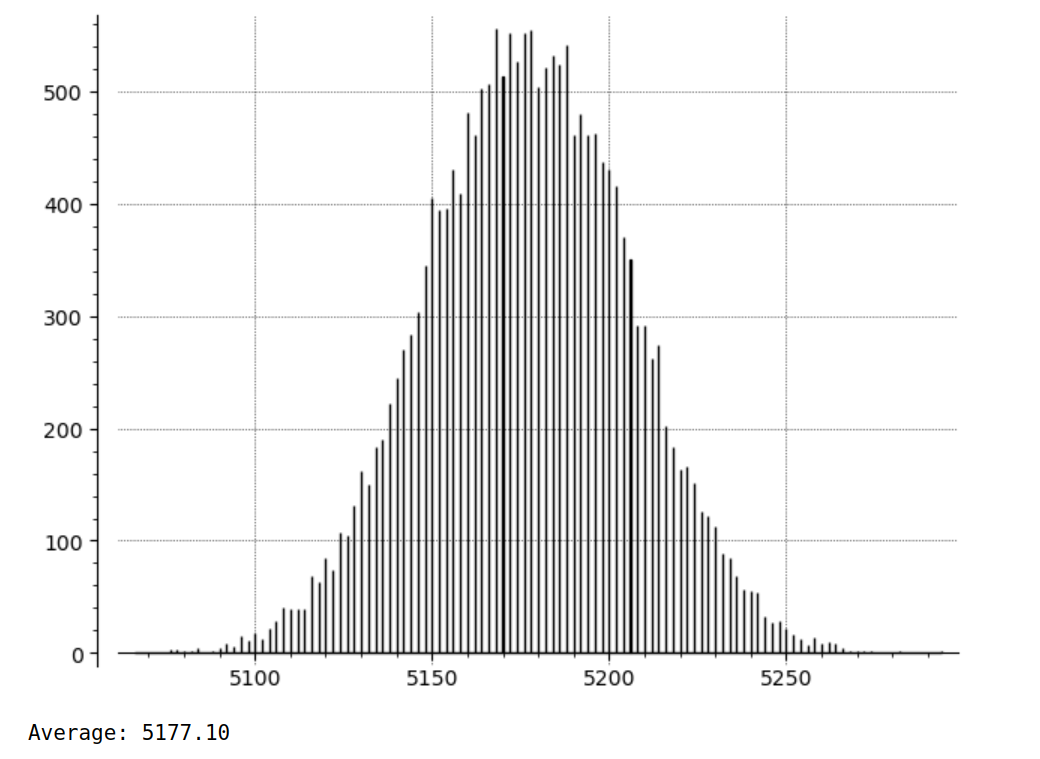}
    \caption{Histogram from $20000$ random samples of\\ Cayley trees of order $n=10000$ }
    \label{fig:histo}
\end{figure}

It is, in fact, quite common to obtain central limit theorems for various graph-theoretic parameters in random trees. In Banderier~et~al.~\cite{banderier2009}, it is shown that the independence number, the path covering number, and the size of the kernel all satisfy a central limit theorem for simply generated trees. In Fuchs~et~al.~\cite{Fuchs2021ANO}, the same results were obtained for the independence number, the domination number, and the clique cover number in random recursive trees and binary search trees, which are other known models of random trees. Similar work on the independence number, the domination number, and the total domination number in conditioned Galton--Watson trees was conducted in the Master's thesis of Rakotoniaina~\cite{rakotoniaina2024}, supervised by the second author.

It is well known that simply generated trees and conditional Galton--Watson trees are essentially equivalent, see \cite[Section 1.2.7]{drmota2009random}. Hence, it might be possible to use the generating function approach as in Banderier~et~al.~\cite{banderier2009} to prove Theorem~\ref{thm:clt}. This is based on a method given in Drmota~\cite{drmota2009random}. However, to use such an analytic-based method, one needs to assume that all moments of $\xi$ are finite (as implied by \cite[Assumption 1]{banderier2009}), and not just the second moment that we assumed here. One advantage of the analytic approach is that it can provide an explicit expression for the constant in the variance. In general, there is no guarantee that such an expression will be useful. For the parameters considered in \cite{banderier2009}, it is not even known if these constants are nonzero, see the remark on \cite[Page~49]{banderier2009}.

Here, we chose to use a probabilistic approach. While we cannot provide an explicit formula for the constant in the variance, we are at least able to prove that, for our parameter, this constant is never zero, using an argument already employed in \cite{Naina}. Hence, we guarantee a nondegenerate Gaussian limit law. The same argument can certainly be used to address the problem stated in the remark on \cite[Page~49]{banderier2009}.

For future work, it would also be interesting to explore other models of random trees. The results in Fuchs~et~al.~\cite{Fuchs2021ANO} were based on a general result in Holmgren and Janson~\cite{holmgren2015limit}. Since the toll function $\varphi$ for the paired domination number is bounded, the same theorem applies in our case as well for the binary search trees and the recursive trees. However, more work needs to be done for other models of random trees, such as $d$-ary increasing trees or even P\'olya trees. The general results in \cite{ralaivaosaona2019increasing, wagner2015central} for these models do not apply to the paired domination number or most of the graph-theoretic parameters mentioned above.

The SageMath code related to this project is available in the GitHub repository~\cite{ralaivaosaona2025paireddomination}.

\section{Acknowledgements}

The first author, Michael A. Henning, acknowledges support from the CoE-MaSS under graph number 2025-014-GRA-Structural Graph Theory. Opinions expressed and conclusions arrived at by the first author are those of the author and not necessarily to be attributed to the CoE-MaSS.

\bibliographystyle{plain}  
\bibliography{references}
\end{document}